\documentclass{amsart}
\usepackage{graphicx, color}
\usepackage{amscd}
\usepackage{amsmath,empheq}
\usepackage{amsfonts}
\usepackage{amssymb}
\usepackage{mathrsfs}
\usepackage[all]{xy}
\newtheorem{theorem}{Theorem}

\newtheorem{prop}[theorem]{Proposition}
\newtheorem{remark}{Remark}

\newenvironment{proof-sketch}{\noindent{\bf Sketch of Proof}\hspace*{1em}}{\qed\bigskip}

\everymath{\displaystyle}
\newcommand{\RR}{\mathbb R}
\newcommand{\NN}{\mathbb N}

\renewcommand{\leq}{\leqslant}

\renewcommand{\geq}{\geqslant}
\begin{document}

\title[Double-phase problems]{Double-phase problems and a discontinuity property of the spectrum}
\author[N.S. Papageorgiou]{Nikolaos S. Papageorgiou}
\address[N.S. Papageorgiou]{National Technical University, Department of Mathematics,
				Zografou Campus, Athens 15780, Greece \& Institute of Mathematics, Physics and Mechanics, 1000 Ljubljana, Slovenia}
\email{\tt npapg@math.ntua.gr}
\author[V.D. R\u{a}dulescu]{Vicen\c{t}iu D. R\u{a}dulescu}
\address[V.D. R\u{a}dulescu]{Institute of Mathematics, Physics and Mechanics, 1000 Ljubljana, Slovenia \& Faculty of Applied Mathematics, AGH University of Science and Technology, 30-059 Krak\'ow, Poland \& Institute of Mathematics ``Simion Stoilow" of the Romanian Academy, P.O. Box 1-764, 014700 Bucharest, Romania}
\email{\tt vicentiu.radulescu@imfm.si}
\author[D.D. Repov\v{s}]{Du\v{s}an D. Repov\v{s}}
\address[D.D. Repov\v{s}]{Faculty of Education and Faculty of Mathematics and Physics, University of Ljubljana \& Institute of Mathematics, Physics and Mechanics, 1000 Ljubljana, Slovenia}
\email{\tt dusan.repovs@guest.arnes.si}
\begin{abstract}
We consider a nonlinear eigenvalue problem driven by the sum of $p$ and $q$-Laplacian. We show that the problem has a continuous spectrum. Our result reveals a discontinuity property for the spectrum of a parametric ($p,q$)-differential operator as the parameter $\beta\rightarrow 1^-$.
\end{abstract}
\maketitle

\section{Introduction}

Let $\Omega\subseteq\RR^N$ be a bounded domain with a $C^2$-boundary $\partial\Omega$. In this paper we study the following nonlinear nonhomogeneous eigenvalue problem:
\begin{equation}
		\left\{\begin{array}{ll}
			-\alpha\Delta_pu(z)-\beta\Delta_qu(z)=\lambda|u(z)|^{q-2}u(z)\ \mbox{in}\ \Omega,&\\
			u|_{\partial\Omega}=0,\alpha>0,\beta>0,\lambda>0,1<p,q<\infty,p\neq q.&
		\end{array}\right\}\tag{$P_\lambda$} \label{eqP}
\end{equation}

For every $r\in(1,+\infty)$ by $\Delta_r$ we denote the $r$-Laplace differential operator defined by
$$\Delta_ru={\rm div}\,(|Du|^{r-2}Du)\ \mbox{for all}\ u\in W^{1,r}_{0}(\Omega).$$

Equations driven by the sum of a $p$-Laplacian and of $q$-Laplacian, known as $(p,q)$-equations, arise in many problems of mathematical physics such as particle physics, see Benci, D'Avenia, Fortunato \& Pisani \cite{1} and nonlinear elasticity, see Zhikov \cite{20}. Zhikov \cite{20} introduced models for strongly anisotropic materials in the context of homogenization. For this purpose defined and studied the double phase functional
$$J_{p,q}(u)=\int_{\Omega}\left(|Du|^p+a(z)|Du|^q\right)dz\ 0\leq a(z)\leq M,1<p<q, u\in W^{1,q}_{0}(\Omega),$$
where the modulating coefficient $a(z)$ dictates the geometry of the composite made of two different materials with hardening exponents $p$ and $q$ respectively.

Such problems were studied by Chorfi \& R\u adulescu \cite{2}, Gasinski \& Papageorgiou \cite{4,5},  Marano, Mosconi \& Papageorgiou \cite{10,11},  Mihailescu \& R\u adulescu \cite{12}, Papageorgiou \& R\u adulescu \cite{14, 15},  Papageorgiou, R\u adulescu \& Repov\v{s} \cite{16,16bis,17}, R\u adulescu \& Repov\v{s} \cite{18}, and Yin \& Yang \cite{19}, under different conditions on the data of the problem.

In the present paper, we show that problem \eqref{eqP} has a continuous spectrum which is the half line $(\beta\hat{\lambda}_1(q),+\infty)$, with $\hat{\lambda}_1(q)>0$ being the principal eigenvalue of $(-\Delta_q,W^{1,q}_{0}(\Omega))$. So, for every $\lambda\in(\beta\hat{\lambda}_1(q),+\infty)$, problem \eqref{eqP} admits a nontrivial solution. Our result reveals an interesting fact better illustration in the particular case where
$$1<p<\infty,q=2,p\neq q,\alpha=1-\beta,\beta\in(0,1).$$

Let $L_{\beta}=-(1-\beta)\Delta_pu-\beta\Delta u$ and let $\hat{\sigma}(\beta)$ denote the spectrum of $L_{\beta}$. We have that
$$\hat{\sigma}(\beta)=(\beta\hat{\lambda}_1(q),+\infty)\ \mbox{for}\ \beta\in(0,1).$$

The set function $\beta\rightarrow\hat{\sigma}(\beta)$ is $h$-continuous (Hausdorff continuous) on $\left(0,1\right]$, but at $\beta=0$ exhibits a discontinuity since $L_0=\Delta$ which has a discrete spectrum.

Our approach is based on the use of the Nehari manifold. So, we perform minimization under constraint.

\section{Preliminaries}

Let $r\in(1,+\infty)$. We recall some basic facts about the spectrum of $(-\Delta_r,W^{1,r}_{0}(\Omega))$. So, we consider nonlinear eigenvalue problem
\begin{equation}\label{eq1}
	-\Delta_ru(z)=\hat{\lambda}|u(z)|^{r-2}u(z)\ \mbox{in}\ \Omega,u|_{\partial\Omega}=0
\end{equation}

We say that $\hat{\lambda}$ is an eigenvalue of $(-\Delta_r,W^{1,r}_{0}(\Omega))$ if problem (\ref{eq1}) admits a nontrivial solution $\hat{u}\in W^{1,p}_{0}(\Omega)$, known as an eigenfunction corresponding to the eigenvalue $\hat{\lambda}$. From the nonlinear regularity theory (see, for example, Gasinski \& Papageorgiou \cite[pp. 737-738]{6}), we have that $\hat{u}\in C^1_0(\overline{\Omega})=\left\{u\in C^1(\overline{\Omega}):u|_{\partial\Omega}=0\right\}$. There is a smallest eigenvalue $\hat{\lambda}_1(r)$ which has the following properties:
\begin{itemize}
	\item $\hat{\lambda}_1(r)$ is isolated (that is, there exists $\epsilon>0$ such that the interval $(\hat{\lambda}_1(r),\hat{\lambda}_1(r)+\epsilon)$ contains no eigenvalue of $(-\Delta_r,W^{1,r}_{0}(\Omega))$).
	\item $\hat{\lambda}_1(r)$ is simple (that is, if $\hat{u},\hat{v}$ are eigenfunction corresponding to $\hat{\lambda}_1(r)$, then $\hat{u}=\mu\hat{v}$ with $\mu\in\RR\backslash\{0\}$).
	\item $\hat{\lambda}_1(r)>0$ and admits the following variational characterization
	\begin{equation}\label{eq2}
		\hat{\lambda}_1(r)=inf\left\{\frac{||Du||^r_r}{||u||^r_r}:u\in W^{1,r}_{0}(\Omega),u\neq 0\right\}
	\end{equation}
\end{itemize}

The infimum in (\ref{eq2}) is realized on the corresponding one dimensional eigenspace. The above properties imply that the elements of this eigenspace are in $C^1_0(\overline{\Omega})$ and do not change sign. By $\hat{u}_1(r)$ we denote the positive, $L^r$-normalized (that is, $||\hat{u}_1(r)||_r=1$) eigenfunction corresponding to $\hat{\lambda}_1(r)>0$. We have
$$\hat{u}_1(r)\in C_+=\{u\in C^1_0(\overline{\Omega}):u(z)\geq 0\ \mbox{for all}\ z\in\overline{\Omega}\}.$$

In fact the nonlinear maximum principle (see, for example, Gasinski \& Papageorgiou \cite[p. 738]{5}), implies that
$$\hat{u}_1(r)\in {\rm int}\, C_+=\{u\in C_+:u(z)>0\ \mbox{for all}\ z\in\Omega,\left.\frac{\partial u}{\partial n}\right|_{\partial\Omega}<0\},$$
with $\frac{\partial u}{\partial n}=(Du,n)_{\RR^N}$ being the outward normal derivative of $u$. Note that if $\hat{u}$ an eigenfunction corresponding to an eigenvalue $\hat{\lambda}\neq\hat{\lambda}_1(r)$, then $\hat{u}$ is nodal (that is, sign changing). The Ljusternik-Schnirelmann minimax scheme gives in addition to $\hat{\lambda}_1(r)$ a whole strictly increasing sequence $\{\hat{\lambda}_k(r)\}_{k\in\NN}$ of distinct eigenvalue such that $\hat{\lambda}_k(r)\rightarrow+\infty$. These are called ``variational eigenvalues'' and we do not know if they exhaust the spectrum of $(-\Delta_r,W^{1,r}_{0}(\Omega))$. However, if $r=2$ (linear eigenvalue problem), then the spectrum is the sequence $\{\hat{\lambda}_k(2)\}_{k\in\NN}$ of variational eigenvalues.

Let $r=max\{p,q\}$ and $\lambda>0$. The energy (Euler) functional for problem \eqref{eqP} is defined by
$$\varphi_\lambda(u)=\frac{\alpha}{p}||Du||^p_p+\frac{\beta}{q}||Du||^q_q-\frac{\lambda}{q}||u||^q_q\ \mbox{for all}\ u\in W^{1,r}_{0}(\Omega).$$

Evidently $\varphi_\lambda\in C^1(W^{1,r}_{0}(\Omega),\RR)$.

The Nehari manifold for the functional $\varphi_\lambda$ is the set
$$N_\lambda=\{u\in W^{1,r}_{0}(\Omega):\left\langle \varphi'_\lambda(u),u\right\rangle=0,u\neq 0\}.$$

In what follows, we denote by $\hat{\sigma}(\alpha,\beta)$  the spectrum of
$$u\rightarrow-\alpha\Delta_pu-\beta\Delta_qu\ \mbox{for all}\ u\in W^{1,r}_{0}(\Omega).$$

So, $\lambda\in\hat{\sigma}(\alpha,\beta)$ if and only if problem \eqref{eqP} admits a nontrivial solution $\hat{u}\in C^1_0(\overline{\Omega})$. This solution is an eigenvector for the eigenvalue $\lambda$.

In what follows for every $\tau\in(1,+\infty)$ by $||\cdot||_{1,\tau}$ we denote the norm of $W^{1,\tau}_{0}(\Omega)$. On account of the Poincare inequality, we have
$$||u||_{1,\tau}=||Du||_\tau\ \mbox{for all}\ u\in W^{1,\tau}_{0}(\Omega).$$

Also, by $A_\tau:W^{1,p}_{0}(\Omega)\rightarrow W^{-1,\tau'}_{0}(\Omega)=W^{1,p}_{0}(\Omega)^*\left(\frac{1}{\tau}+\frac{1}{\tau'}=1\right)$ we denote the nonlinear operator defined by
$$\left\langle A_\tau(u),h\right\rangle=\int_{\Omega}|Du|^{\tau-2}(Du,Dh)_{\RR^N}dt\ \mbox{for all}\ u,h\in W^{1,\tau}_{0}(\Omega).$$

This operator is bounded (that is, maps bounded sets to bounded sets), continuous, strictly monotone (hence maximal monotone too).

\section{The spectrum of \eqref{eqP}}

First we deal with the easy case where $1<q<p$. As we will see in the sequel, for this case $\varphi_\lambda(\cdot)$ is coercive and so we can use the direct method of the calculus of variations.
\begin{prop}\label{prop1}
	If $1<q<p$, then $\hat{\sigma}(\alpha,\beta)=(\beta\hat{\lambda}_1(q),+\infty)$ and the eigenvectors belong in $C^1_0(\overline{\Omega})$.
\end{prop}
\begin{proof}
	Now $r=max\{p,q\}=p$. Evidently, if $\lambda\leq\beta\hat{\lambda}_1(q)$, then $\lambda\notin\hat{\sigma}(\alpha,\beta)$ ot other wise we violate (\ref{eq2}).
	
	Let $\lambda>\beta\hat{\lambda}_1(q)$ and $u\in W^{1,p}_{0}(\Omega)$. We have
	\begin{eqnarray*}
		&\varphi_\lambda(u)&\geq\frac{\alpha}{p}||Du||^p_p-\frac{\lambda}{\hat{\lambda}_1(q)q}||Du||^q_q\ (\mbox{see (\ref{eq2})})\\
		&&\geq\frac{\alpha}{p}||Du||^p_p-c_1||Du||^q_q\ \mbox{for some}\ c_1>0\ (\mbox{since}\ q<p),\\
		\Rightarrow&\varphi_\lambda(u)&\geq c_2||u||^p-c_3||u||^q\ \mbox{for some}\ c_2,c_3>0,\\
		\Rightarrow&\varphi_\lambda(\cdot)&\ \mbox{is coercive}\ (\mbox{since}\ q<p).
	\end{eqnarray*}
	
	Also, by the Sobolev embedding theorem $\varphi_\lambda(\cdot)$ is sequentially weakly lower semicontinuous. So, by Weierstrass-Tonelli theorem, we can find $\hat{u}_\lambda\in W^{1,p}_{0}(\Omega)$ such that
	\begin{equation}\label{eq3}
		\varphi_\lambda(\hat{u}_\lambda)=inf\{\hat{\varphi}_\lambda(u):u\in W^{1,p}_{0}(\Omega)\}.
	\end{equation}
	
	For $t>0$ we have
	\begin{eqnarray*}
		\varphi_\lambda(t\hat{u}_1(q))&=&\frac{t^p\alpha}{p}||D\hat{u}_1(q)||^p_p+\frac{t^q}{q}[\beta\hat{\lambda}_1(q)-\lambda]\ (\mbox{recall that}\ ||\hat{u}_1(q)||_q=1)\\
		&=&c_4t^p-c_5t^q\ \mbox{for some}\ c_4,c_5>0\ (\mbox{recall that}\ \lambda>\beta\hat{\lambda}_1(q))
	\end{eqnarray*}
	
	Since $q<p$, choosing $t\in(0,1)$ small we have
	\begin{eqnarray*}
		&&\varphi_\lambda(t\hat{u}_1(q))<0,\\
		&\Rightarrow&\varphi_\lambda(\hat{u}_\lambda)<0=\varphi_\lambda(0)\ (\mbox{see (\ref{eq3})}),\\
		&\Rightarrow&\hat{u}_\lambda\neq 0.
	\end{eqnarray*}
	
	From (\ref{eq3}) we have
	\begin{eqnarray*}
		&&\varphi'_\lambda(\hat{u}_\lambda)=0,\\
		&\Rightarrow&\left\langle \alpha A_p(\hat{u}_\lambda),h\right\rangle+\left\langle \beta A_q(\hat{u}_\lambda),h\right\rangle=\lambda\int_{\Omega}|\hat{u}_\lambda|^{q-2}\hat{u}_\lambda hdz\ \mbox{for all}\ h\in W^{1,p}_{0}(\Omega),\\
		&\Rightarrow&-\alpha\Delta\hat{u}_\lambda(z)-\beta\Delta_q\hat{u}_\lambda(z)=\lambda|\hat{u}_\lambda(z)|^{q-2}\hat{u}_\lambda(z)\ \mbox{for almost all}\ z\in\Omega,\hat{u}_\lambda|_{\partial\Omega}=0,\\
		&\Rightarrow&\hat{u}_\lambda\in C^1_0(\overline{\Omega})\ (\mbox{by the nonlinear regularity theory, see Lieberman \cite{9}}).
	\end{eqnarray*}
\end{proof}

When $1<p<q$, the energy functional is no longer coercive. So, the direct method of the calculus of the variations fails and we have to use a different approach. Instead we will minimize $\varphi_\lambda$ on the Nehair manifold $N_\lambda$.

First we show that $N_\lambda\neq\varnothing$.
\begin{prop}\label{prop2}
	$\lambda>\beta\hat{\lambda}_1(q)$ if and only if $N_\lambda\neq\varnothing$.
\end{prop}
\begin{proof}
	As before (see the proof of Proposition \ref{prop1}), using (\ref{eq2}) we see that
	$$N_\lambda\neq\varnothing\Rightarrow\lambda>\beta\hat{\lambda}_1(q).$$
	
	Now suppose that $\lambda>\beta\hat{\lambda}_1(q)$. Then on account of (\ref{eq2}) we can find $u\in W^{1,q}_{0}(\Omega)\ u\neq 0$ such that
	\begin{equation}\label{eq4}
		||Du||^q_q<\frac{\lambda}{\beta}||u||^q_q
	\end{equation}
	
	Consider the function $\xi_\lambda:(0,+\infty)\rightarrow\RR$ defined by
	\begin{eqnarray}\label{eq5}
		\xi_\lambda(t)&=&\left\langle \varphi'_\lambda(tu),tu\right\rangle\nonumber\\
		&=&\alpha t^p||Du||^p_p+\beta t^q||Du||^q_q-\lambda t^q||u||^q_q\nonumber\\
		&=&t^p\alpha||Du||^p_p+t^q[\beta||Du||^q_q-\alpha||u||^q_q]\nonumber\\
		&=&c_6t^p-c_7t^q\ \mbox{for some}\ c_6,c_7>0\ (\mbox{see (\ref{eq4})}).
	\end{eqnarray}
	
	Since $q>p$, from (\ref{eq5}) we see that
	$$\xi_\lambda(t)\rightarrow-\infty\ \mbox{as}\ t\rightarrow+\infty$$
	
	On the other hand for $t\in(0,1)$ small, we have
	$$\xi_\lambda(t)>0\ \mbox{see (\ref{eq5})}.$$
	
	Therefore, by Balzano's theorem, we can find $t_0>0$ such that
	\begin{eqnarray*}
		&&\xi_\lambda(t_0)=0,\\
		&\Rightarrow&\left\langle \varphi'_\lambda(t_0u),t_0u\right\rangle=0\ \mbox{with}\ t_0u\neq 0,\\
		&\Rightarrow&t_0u\in N_\lambda\ \mbox{and so}\ N_\lambda\neq\varnothing.
	\end{eqnarray*}
\end{proof}

We define
\begin{equation}\label{eq6}
	m_\lambda=inf\{\varphi_\lambda(u):u\in N_\lambda\}.
\end{equation}

For $u\in N_\lambda$, we have
\begin{equation}\label{eq7}
	\alpha||Du||^p_p+\beta||Du||^q_q=\lambda||u||^q_q.
\end{equation}

Therefore
\begin{eqnarray}\label{eq8}
	&\varphi_\lambda(u)&=\frac{\alpha}{p}||Du||^p_p+\frac{\beta}{q}||Du||^q_q-\frac{1}{q}[\alpha||Du||^p_p+\beta||Du||^q_q]\ \mbox{(see (\ref{eq7}))}\nonumber\\
	&&=\alpha\left[\frac{1}{p}-\frac{1}{q}\right]||Du||^p_p,\\
	\Rightarrow&m_\lambda\geq 0&(\mbox{see (\ref{eq6})})\nonumber
\end{eqnarray}

From (\ref{eq8}) we infer that $\varphi_\lambda|_{N_\lambda}$ is coercive on $W^{1,p}_{0}(\Omega)$.
\begin{prop}\label{prop3}
	If $\lambda>\beta\hat{\lambda}_1(q)$, then every minimizing sequence of (\ref{eq6}) is bounded in $W^{1,q}_{0}(\Omega)$.
\end{prop}
\begin{proof}
	We argue by contradiction. So, suppose that $\{u_n\}_{n\geq 1}\subseteq W^{1,q}_{0}(\Omega)$ is a minimizing sequence of (\ref{eq6}) such that
	$$||u_n||_{1,q}\rightarrow+\infty.$$
	
	We have
	\begin{eqnarray}
		&&\alpha||Du_n||^p_p+\beta||Du_n||^q_q=\lambda||u_n||^q_q\ \mbox{for all}\ n\in\NN,\label{eq9}\\
		&\Rightarrow&\beta||Du_n||^q_q=\beta||u_n||^q_{1,q}\leq\lambda||u_n||^q_q\ \mbox{for all}\ n\in\NN,\label{eq10}\\
		&\Rightarrow&||u_n||_q\rightarrow+\infty\ \mbox{as}\ n\rightarrow \infty.\nonumber
	\end{eqnarray}
	
	We set $y_n=\frac{u_n}{||u_n||_q}\ n\in\NN$. Then $||y_n||_q=1\ n\in\NN$. Also, from (\ref{eq10}) we have
	\begin{eqnarray*}
		&&||Dy_n||^q_q\leq\frac{\lambda}{\beta}||y_n||^q_q=\frac{\lambda}{\beta}\ \mbox{for all}\ n\in\NN,\\
		&\Rightarrow&\{y_n\}_{n\geq 1}\subseteq W^{1,q}_{0}(\Omega)\ \mbox{is bounded}.
	\end{eqnarray*}
	
	So, we may assume that
	\begin{equation}\label{eq11}
		y_n\stackrel{w}{\rightarrow}y\ \mbox{in}\ W^{1,q}_{0}(\Omega)\ \mbox{and}\ y_n\rightarrow y\ \mbox{in}\ L^q(\Omega).
	\end{equation}
	
	We multiply (\ref{eq9}) with $\frac{1}{||u_n||^q_q}$. We obtain
	\begin{equation}\label{eq12}
		\alpha||Dy_n||^p_p=\frac{\lambda||u_n||^q_q-\beta||Du_n||^q_q}{||u_n||^p_q},\ n\in\NN.
	\end{equation}
	
	Recall that $\{u_n\}_{n\geq 1}\subseteq N_\lambda$ is a minimizing sequence for (\ref{eq6}). So, we have
	\begin{equation}\label{eq13}
		\left(\frac{1}{p}-\frac{1}{q}\right)[\lambda||u_n||^q_q-\beta||Du_n||^q_q]\rightarrow m_\lambda\ \mbox{as}\ n\rightarrow\infty\ (\mbox{see (\ref{eq7}), (\ref{eq8})})
	\end{equation}
	
	Using (\ref{eq13}) in (\ref{eq12}), we infer that
	\begin{eqnarray*}
		&&y_n\rightarrow 0\ \mbox{in}\ W^{1,p}_{0}(\Omega),\\
		&\Rightarrow&y_n\rightarrow 0\ \mbox{in}\ L^q(\Omega)\ (\mbox{see (\ref{eq11})}),
	\end{eqnarray*}
	a contradiction since $||y_n||_q=1$ for all $n\in\NN$.
	
	Therefore we conclude that every minimizing sequence of (\ref{eq6}) is bounded in $W^{1,q}_{0}(\Omega)$.
\end{proof}

We have already seen that $m_\lambda\geq 0$. We can say more.
\begin{prop}\label{prop4}
	If $\lambda>\beta\hat{\lambda}_1(q)$, then $m_\lambda>0.$
\end{prop}
\begin{proof}
	Arguing by contradiction, suppose that $m_\lambda=0$. Then we can find $\{u_n\}_{n\geq 1}\subseteq N_\lambda$ such that $\varphi_{\lambda_-}(u_n)\rightarrow 0^+$. From (\ref{eq8}) we have
	\begin{eqnarray}\label{eq14}
		&&\alpha\left[\frac{1}{p}-\frac{1}{q}\right]||Du_n||^p_p\rightarrow 0,\nonumber\\
		&\Rightarrow&u_n\rightarrow 0\ \mbox{in}\ W^{1,p}_{0}(\Omega)
	\end{eqnarray}
	
	Then from (\ref{eq14}) and Proposition \ref{prop3}, we infer that
	\begin{equation}\label{eq15}
		u_n\stackrel{w}{\rightarrow}0\ \mbox{in}\ W^{1,q}_{0}(\Omega)\ \mbox{and}\ u_n\rightarrow 0\ \mbox{in}\ L^q(\Omega)
	\end{equation}
	
	From (\ref{eq14}), (\ref{eq15}) and (\ref{eq7}), it follows that
	\begin{eqnarray}\label{eq16}
		&&||Du_n||_q\rightarrow 0,\nonumber\\
		&\Rightarrow&u_n\rightarrow 0\ \mbox{in}\ W^{1,q}_{0}(\Omega).
	\end{eqnarray}
	
	Let $v_n=\frac{u_n}{||u_n||_q}\ n\in\NN$. Then $||v_n||_q=1\ n\in\NN$. We have
	\begin{eqnarray}
		&&\lambda||v_n||^q_q-\beta||Dv||^q_q=\frac{\alpha}{||u_n||^{q-p}_q}||Dv_n||^p_p>0\ \mbox{for all}\ n\in\NN,\label{eq17}\\
		&\Rightarrow&||Dv_n||^q_q\leq\frac{\lambda}{\beta}\ \mbox{for all}\ n\in\NN,\nonumber\\
		&\Rightarrow&\{v_n\}_{n\geq 1}\subseteq W^{1,q}_{0}(\Omega)\ \mbox{is bounded}\label{eq18}
	\end{eqnarray}
	
	Then from (\ref{eq17}) and (\ref{eq18}) it follows that
	\begin{eqnarray}\label{eq19}
		&&\frac{\alpha}{||u_n||^{q-p}_q}||Dv_n||^p_p\leq c_8\ \mbox{for some}\ c_8>0,\ \mbox{all}\ n\in\NN,\nonumber\\
		&\Rightarrow&||Dv_n||_p\rightarrow 0\ (\mbox{see (\ref{eq16}) and recall that}\ p<q),\nonumber\\
		&\Rightarrow&v_n\rightarrow 0\ \mbox{in}\ W^{1,p}_{0}(\Omega)
	\end{eqnarray}
	
	From (\ref{eq18}) and (\ref{eq19}), we infer that
	$$v_n\rightarrow 0\ \mbox{in}\ L^q(\Omega),$$
	a contradiction since $||v_n||_q=1\ n\in\NN$. From this we conclude that $m_\lambda>0$.
\end{proof}

\begin{prop}\label{prop5}
	If $\lambda>\beta\hat{\lambda}_1(q)$, then there exists $\hat{u}_\lambda\in N_\lambda$ such that $m_\lambda=\varphi_\lambda(\hat{u}_\lambda).$
\end{prop}
\begin{proof}
	Let $\{u_n\}_{n\geq 1}\subseteq N_\lambda$ such that $\varphi_\lambda(u_n)\rightarrow m_\lambda$. According to Proposition \ref{prop3}, $\{u_n\}_{n\geq 1}\subseteq W^{1,q}_{0}(\Omega)$ is bounded. So, we may assume that
	\begin{equation}\label{eq20}
		u_n\stackrel{w}{\rightarrow}\hat{u}_\lambda\ \mbox{in}\ W^{1,q}_{0}(\Omega)\ \mbox{and}\ u_n\rightarrow\hat{u}_\lambda\ \mbox{in}\ L^q(\Omega)
	\end{equation}
	
	Since $u_n\in N_\lambda\ n\in\NN$, we have
	\begin{equation}\label{eq21}
		\alpha||Du_n||^p_p+\beta||Du_n||^q_q=\lambda||u_n||^q_q\ \mbox{for all}\ n\in\NN.
	\end{equation}
	
	Passing to the limit as $n\rightarrow\infty$ and using (\ref{eq20}) and the weak lower semicontinuity of the functional in a Banach space, we obtain
	\begin{equation}\label{eq22}
		\alpha||D\hat{u}_\lambda||^p_p\leq\lambda||\hat{u}_\lambda||^q_q-\beta||D\hat{u}_\lambda||^q_q.
	\end{equation}
	
	Note that $\lambda||\hat{u}_\lambda||^q_q-\beta||D\hat{u}_\lambda||^q_q\neq 0$ or otherwise from (\ref{eq18}), we have
	\begin{eqnarray*}
		&&||Du_n||_p\rightarrow 0,\\
		&\Rightarrow&u_n\rightarrow 0\ \mbox{in}\ W^{1,p}_{0}(\Omega)
	\end{eqnarray*}
	
	Recall that
	$$\varphi_\lambda(u_n)=\alpha\left[\frac{1}{p}-\frac{1}{q}\right]||Du_n||^p_p\ \mbox{for all}\ n\in\NN\ (\mbox{see (\ref{eq8})}).$$
	
	So, it follows that
	\begin{eqnarray*}
		&&\varphi_\lambda(u_n)\rightarrow 0\ \mbox{as}\ n\rightarrow\infty,\\
		&\Rightarrow&m_\lambda=0,
	\end{eqnarray*}
	which contradicts Proposition \ref{prop4}. Therefore
	\begin{eqnarray*}
		&&\lambda||\hat{u}_\lambda||^q_q-\beta||D\hat{u}_\lambda||^q_q\neq 0,\\
		&\Rightarrow&\hat{u}_\lambda\neq 0.
	\end{eqnarray*}
	
	Also, exploiting the sequential weak lower semicontinuity of $\varphi_\lambda(\cdot)$, we have
	$$\varphi_\lambda(\hat{u}_\lambda)\leq\lim\limits_{n\rightarrow\infty}\varphi_\lambda(u_n)=m_\lambda\ (\mbox{see (\ref{eq20})}).$$
	
	If we show that $\hat{u}_\lambda\in N_\lambda$, then $\varphi_\lambda(\hat{u}_\lambda)=m_\lambda$ and this will conclude the proof.
	
	To this end, let
	$$\hat{\xi}_\lambda(t)=\left\langle \varphi'_\lambda(t\hat{u}_\lambda,t\hat{u}_\lambda)\right\rangle\ \mbox{for all}\ t\in[0,1].$$
	
	Evidently $\hat{\xi}_\lambda(\cdot)$ is a continuous function. Arguing indirectly, suppose that $\hat{u}_\lambda\notin N_\lambda$. Then since $u_n\in N_\lambda$ for all $n\in\NN$, from (\ref{eq20}) we infer that
	\begin{equation}\label{eq23}
		\hat{\xi}_\lambda(1)<0.
	\end{equation}
	
	On the other hand, note that since $\lambda>\beta\hat{\lambda}_1(q)$, we have
	\begin{eqnarray}\label{eq24}
		&&\hat{\xi}_\lambda(t)\geq c_9t^p-c_{10}t^q\ \mbox{for some}\ c_9,c_{10}>0,\nonumber\\
		&\Rightarrow&\hat{\xi}_\lambda(t)>0\ \mbox{for all}\ t\in(0,\epsilon)\ \mbox{with}\ \epsilon\in(0,1)\ \mbox{small}\ (\mbox{recall}\ p<q).
	\end{eqnarray}
	
	From (\ref{eq23}), (\ref{eq24}) and Bolzano's theorem, we see that there exists $t^*\in(0,1)$ such that
	\begin{eqnarray*}
		&&\hat{\xi}_\lambda(t^*\hat{u}_\lambda)=0,\\
		&\Rightarrow&t^*\hat{u}_\lambda\in N_\lambda.
	\end{eqnarray*}
	
	Then using (\ref{eq8}) we have
	\begin{eqnarray*}
		m_\lambda\leq\varphi_\lambda(t^*\hat{u}_\lambda)&=&\alpha\left[\frac{1}{p}-\frac{1}{q}\right](t^*)^p||D\hat{u}_\lambda||^p_p\\
		&<&\alpha\left[\frac{1}{p}-\frac{1}{q}\right]||D\hat{u}_\lambda||^p_p\ (\mbox{since}\ t^*\in(0,1))\\
		&\leq&\alpha\left[\frac{1}{p}-\frac{1}{q}\right]\liminf\limits_{n\rightarrow\infty}||Du_n||^p_p\ (\mbox{see (\ref{eq20})})\\
		&=&m_\lambda,
	\end{eqnarray*}
	a contradiction. Therefore $\hat{u}_\lambda\in N_\lambda$ and this finishes the proof.
\end{proof}

So, we can state the following theorem concerning problem \eqref{eqP}.
\begin{theorem}\label{th6}
	If $\lambda>\beta\hat{\lambda}_1(q)$ then $\lambda$ is an eigenvalue of problem \eqref{eqP} with eigenfunction $\hat{\lambda}\in C^1_0(\overline{\Omega})$.
\end{theorem}
\begin{proof}
	For $1<q<p$, this follows from Proposition \ref{prop1}.
	
	For $1<q<p$, let $h\in W^{1,q}_{0}(\Omega)$. Choose $\epsilon>0$ such that $\hat{u}_\lambda+sh\not\equiv 0$ for $s\in(-\epsilon,\epsilon)$. We set
	$$t(s)=\left[\frac{\lambda||\hat{u}_\lambda+sh||^q_q-\beta||D(\hat{u}_\lambda+sh)||^q_q}{\alpha||D(\hat{u}_\lambda+sh)||^p_p}\right]^{\frac{1}{p-q}},\ s\in(-\epsilon,\epsilon).$$
	
	Then we have that $s\rightarrow t(s)$ is a curve in $N_\lambda$ and it is differentiable. Let $\hat{\xi}_\lambda:(-\epsilon,\epsilon)\rightarrow\RR$ be defined by
	$$\hat{\xi}_\lambda(s)=\varphi_\lambda(t(s)(\hat{u}_\lambda+sh)),\ s\in(-\epsilon,\epsilon).$$
	
	Evidently $s=0$ is a minimizer of $\hat{\xi}_\lambda(\cdot)$ and so
	\begin{eqnarray*}
		0&=&\hat{\xi}_\lambda(0)\\
		&=&\left\langle \varphi'_\lambda(\hat{u}_\lambda),t'(0)\hat{u}_\lambda+h\right\rangle\ (\mbox{by the chain rule})\\
		&=&t'(0)\left\langle \varphi'_\lambda(\hat{u}_\lambda),\hat{u}_\lambda\right\rangle+\left\langle \varphi'_\lambda(\hat{u}_\lambda),h\right\rangle\\
		&=&\left\langle \varphi'_\lambda(\hat{u}_\lambda),h \right\rangle\ (\mbox{since}\ \hat{u}_\lambda\in N_\lambda),\\
		\Rightarrow&&\alpha\left\langle A_p(\hat{u}_\lambda),h\right\rangle+\beta\left\langle A_q(\hat{u}_\lambda),h\right\rangle=\lambda\int_{\Omega}|\hat{u}_\lambda|^{q-2}\hat{u}_\lambda hzd,\\
		\Rightarrow&&-\alpha\Delta_p\hat{u}_\lambda(z)-\beta\Delta_q\hat{u}_\lambda(z)=\lambda|\hat{u}_\lambda(z)|^{q-2}\hat{u}_\lambda(z)\ \mbox{for almost all}\ z\in\Omega,\ \hat{u}_\lambda|_{\partial\Omega}=0,\\
		\Rightarrow&&\hat{u}_\lambda\neq 0\ \mbox{is an eigenfunction with eigenvalue}\ \lambda>\beta\hat{\lambda}_1(q).
	\end{eqnarray*}
	
	Then nonlinear regularity theory of Lieberman \cite[p. 320]{9}, implies that $\hat{u}_\lambda\in C^1_0(\overline{\Omega})$.
\end{proof}

\begin{remark}
	In the terminology of critical point theory, the above proof shows that the Nehari manifold, is a natural constant for the functional $\varphi_\lambda$ (see Gasinski \& Papageorgiou \cite[p. 812]{6}).
\end{remark}

Now suppose that $\alpha=(1-\beta),\beta\in(0,1)$. Let $L_\beta=-(1-\beta)\Delta_p-\beta\Delta_q$ and let $\hat{\sigma}(\beta)$ be the spectrum of $L_\beta$. From Theorem \ref{th6}, we know that
$$\hat{\sigma}(\beta)=(\beta\hat{\lambda}_1(q),+\infty).$$

Evidently $\hat{\sigma}(\cdot)$ is Hausdorff and Vietoris continuous on $(0,1)$ (see Hu \& Papageorgiou \cite{7}), but at $\beta=1$, it exhibits a discontinuity since
$$\hat{\sigma}(1)=\mbox{the spectrum of}\ (-\Delta_q,W^{1,q}_{0}(\Omega))$$
and from Section 2, we know that $\hat{\lambda}_1(q)>0$ is isolated and so $\hat{\sigma}(1)\neq(\hat{\lambda}_1(q),+\infty)$. This is more emphatically illustrated when $q=2$. Then
$$\hat{\sigma}(\beta)=(\beta\hat{\lambda}_1(2),+\infty)\ \mbox{for all}\ \beta\in(0,1)$$
but at $\beta=1$, we have
$$\hat{\sigma}(1)=\{\hat{\lambda}_k(2)\}_{k\geq 1}\ (\mbox{discrete spectrum}).$$	
	
\medskip
{\bf Acknowledgments.} This research was supported by the Slovenian Research Agency grants
P1-0292, J1-8131, J1-7025, N1-0064, and N1-0083. V.D.~R\u adulescu acknowledges the support through a grant of the Romanian Ministry of Research and Innovation, CNCS-UEFISCDI, project number PN-III-P4-ID-PCE-2016-0130,
within PNCDI III.

\end{document}